\documentclass[10pt,english]{article}
\usepackage[utf8]{inputenc}
\usepackage{setspace}
\usepackage[english]{babel}
\usepackage{comment}
\usepackage{amsfonts}
\usepackage{amsthm}
\usepackage{graphicx}
\usepackage[mathscr]{euscript}
\usepackage{hyperref}
\usepackage{xcolor} 
\usepackage{bm}
\usepackage{amsthm, thmtools}
\usepackage{hyperref}
\usepackage{nameref}
\usepackage{amsmath, amssymb, bbm}

\usepackage{
hyperref,
cleveref,
}

\newtheorem{theorem}{Theorem}[section]
\newtheorem{remark}[theorem]{Remark}

\newcommand{\Row}{\text{Row}}
\newcommand{\Null}{\text{Null}}
\newcommand{\Span}{\text{Span}}

\newcommand{\Z}{\mathbb{Z}}
\newcommand{\G}{\bm{G}}
\newcommand{\bb}[2]{\bm{#1}_{#2}}
\theoremstyle{definition}
\newtheorem{definition}[theorem]{Definition}

\begin{document}

\pagenumbering{arabic}

\author{
  Jessica Wang \\
  \texttt{jwang22@math.wpi.edu}
  \and
  Joseph D. Fehribach \\
  \texttt{bach@math.wpi.edu}
}

\title
{
Prime, Composite and Fundamental \\ Kirchhoff Graphs
}

\date{}\maketitle

\begin{abstract}
A Kirchhoff graph is a vector graph
with orthogonal cycles and vertex cuts.
An algorithm has been developed
that constructs all the Kirchhoff graphs up to a fixed edge multiplicity.
This algorithm is used to explore the structure of prime Kirchhoff graph tilings.
The existence of infinitely many prime Kirchhoff graphs given a set of fundamental Kirchhoff graphs
is established,
as is the existence of a minimal multiplicity for Kirchhoff graphs to exist.
\end{abstract}

\section{Introduction}\label{Introduction}

In recent years,
Kirchhoff graphs have been studied extensively
by the second author and his colleagues and students
(see Fehribach~\cite{SIAM09JDF,AMC15JDF}, Fehribach \& McDonald~\cite{CM18FM},
Reese, Fehribach, Paffenroth \& Servatius~\cite{Symmetry16RFP,LAA18RFPS,LAA19RFPS,CM20JDF,LAA22RFP}).
Previous studies have considered the construction and properties of individual Kirchhoff graphs.
The present work considers the structure of families of Kirchhoff graphs
all of which are associated with the same set of edge vectors.
In order to do this,
we have developed a numerical method
for constructing all Kirchhoff graphs up to a certain size
for a given edge vector set.
The results of our computations
allow the definition of prime, composite and fundamental Kirchhoff graphs,
motivate the proofs of several results,
and lead to a number of open questions
regarding the structure of these Kirchhoff graph families.

A Kirchhoff graph is a vector graph
whose edges are vectors (or whose edges are assigned vectors)
that satisfy an orthogonality condition between its cycles and its vertex cuts.
There is a cycle in the graph only when the corresponding vectors add to zero in the vector space.
Consider a set $S:=\{\bm{s}_1,\bm{s}_2,\ldots, \bm{s}_n\}$ of vectors
in a vector space $\mathcal{V}$ over $\mathbb{Q}$.
These are the edge vectors for our vector graphs.
For simplicity, suppose that no vector in $S$ is a scalar multiple of another,,
and suppose there is a $k$
where $1<k<n$ so that $\{\bm{s}_1,\bm{s}_2,\ldots,\bm{s}_k\}$ forms a basis
for $\Span(S)$.
Then for $[\bm{s}_1,\;\bm{s}_2,\;\ldots\;\bm{s}_n]$, a row vector of vectors,
there is a coefficient matrix $C'$ such that $[\bm{s}_1,\;\bm{s}_2,\;\ldots\;\bm{s}_n]\cdot[C'/-I_{n-k}]=0$
where $[C'/-I_{n-k}]$ is a block matrix with $C'$ over $-I_{n-k}$.
Let $q\in\Z^+$ be the least common multiple of denominators of $C'$,
and define $C:=qC'$
so the entries of $C$ are integers.
Then define $N:=[C/-qI_{n-k}]$ as the null matrix for $S$,
and $R:=[qI_k|C]$ as the row matrix for $S$.
Specifically, the columns of $N$ form a basis for $\Null(R)$,
and the columns of $R$ can be used to represent the vectors in $S$
since all finite-dimensional vector spaces
over a given field
are isomorphic.
This means that any matrix $A$ that is row equivalent to $R$
has the same row space and null space as $R$,
a the same set of Kirchhoff graphs.

The orthogonality condition mentioned above
corresponds to orthocomplementary of the matrices $R$ and $N$.
For a vertex $v$ in a vector graph $\bm{G}$,
the \textbf{vertex cut} of $v$,
denoted $\bm{\lambda}(v)=\{\lambda_1,\ldots,\lambda_n\}$,
has entries corresponding to the vectors $\bm{s}_1,\ldots,\bm{s}_n$.
For each $i$,
entry $\lambda_i$ is the net number of times $\bm{s}_i$ exits vertex $v$.
Add $1$ to $\lambda_i$ for each copy of $\bm{s}_i$ that exits $v$;
subtract $1$ from $\lambda_i$ for each copy of $\bm{s}_i$ that enters $v$.
Thus $\lambda_i$ zero if $\bm{s}_i$ is not incident on $v$,
or if $\bm{s}_i$ enters and exits the same number of times.
A \textbf{cycle} $C$ in a vector graph $\G$ is an alternating sequence
of vertices and edges that starts and ends with the same vertex
in which no vertex appears twice except for the first and the last.
Cycles in a vector graph corresponds to linear combinations
of the edge vectors $\bm{s}_1,\ldots,\bm{s}_n$ that add to the zero vector.
The \textbf{cycle vector} for a cycle $C$,
denoted $\bm{\chi}(C)=\{\chi_1,\ldots,\chi_n\}$,
has entries corresponding to vectors $\bm{s}_1,\ldots,\bm{s}_n$.
For each $i$, entry $\chi_i$ is the net number of times $\bm{s}_i$ appears in the cycle.
Add $1$ to the $i$-th component
each time $C$ traverses an $\bm{s}_i$ in the forward direction,
and subtract $1$ for each $\bm{s}_i$ in the backward direction. 
A {\bf Kirchhoff graph}
for a set of edge vectors $S$
is then vector graph satisfying two conditions:
\begin{enumerate}
  \item For each vertex $v$ of $\bm{G}$,
        $\bm{\lambda}(v)\in\Row(R)$.
  \item For each cycle $C$ of $\bm{G}$,
        $\bm{\chi}(C)\in\Null(R)$,
        and there is a cycle basis for the cycle space of $\bm{G}$
        corresponding to a basis for $\Null(R)$.
\end{enumerate}
This implies that for each vertex and each cycle of $\bm{G}$,
$\bm{\lambda}(v)\perp\bm{\chi}(C)$.

As a simple example
of Kirchhoff graphs,
consider the matrix
$$
R_1 =
\left[\begin{array}{rrrr}
~2 & ~0 & ~1 & ~1 \\
 0 &  2 &  1 & -1
\end{array}\right]
$$
with $S_1=\{\bm{s}_1,\bm{s}_2,\bm{s}_3,\bm{s}_4\}$
and the columns of $R_1$ being a representation of the vectors of $S$.
Then two Kirchhoff graphs
for $R_1$ and $S_1$
are given in Figure~\ref{fig:F1}.
\begin{figure}[htp]
    \centering
    \includegraphics[width=5cm]{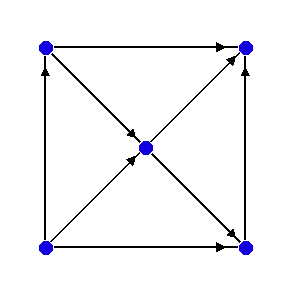}
    \includegraphics[width=6cm]{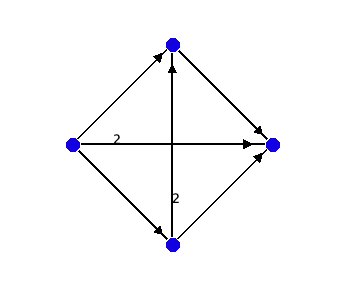}
    \caption{Kirchhoff graphs $\bm{F}_1$ (left) and $\bm{F}_2$ (right)
       for the edge vectors
       $\bm{s}_1 = [2, 0]^{\mbox{\scriptsize T}}$,
       $\bm{s}_2 = [0, 2]^{\mbox{\scriptsize T}}$,
       $\bm{s}_3 = [1, 1]^{\mbox{\scriptsize T}}$,
       $\bm{s}_4 = [1, -1]^{\mbox{\scriptsize T}}$
       embedded in the Euclidean plane.
       The small 2 on $\bm{s}_1$ and $\bm{s}_2$
       in $\bm{F}_2$ indicate that two copies
       of these edge vectors
       lie on top of each other in this Kirchhoff graph.
       Notice that their vertex cuts lie in the row space,
       and the cycles form a basis for the null space.}
    \label{fig:F1}
\end{figure}

Notice that there are two copies of each of the edge vectors
in each of the Kirchhoff graphs in Figure~\ref{fig:F1}.
A Kirchhoff graph is {\bf uniform} if and only if
each of its edge vectors appear the same number of times.
In addition,
a vector graph is {\bf vector 2-connected} if and only if
for any pair of vector edges
$\bm{s}_i$ and $\bm{s}_j$,
there exists a cycle $C$
such that the cycle vector $\bm{\chi}(C)$
is nonzero with respect to both
$\bm{s}_i$ and $\bm{s}_j$.
Reese, Fehribach, Paffenroth \& Servatius~\cite{LAA19RFPS, CM20JDF}
proved the following:
\begin{theorem}
\label{thm:uniform}
Every vector 2-connected Kirchhoff graph is uniform.
\end{theorem}

\noindent
One important way that a Kirchhoff graph will fail to be vector 2-connected
is if the matrix $C$ has a row of zeros.
All Kirchhoff graphs considered here are vector 2-connected and hence uniform.
Given that Kirchhoff graph is vector 2-connected and uniform,
let $m=m(\G)$ be the {\bf edge multiplicity}
(or simply the {\bf multiplicity})
of $\bm{G}$,
the number of times each edge vector appears in $\bm{G}$.

A second important property of Kirchhoff graphs is {\em chirality}:
Given a Kirchhoff graph $\bm{G}$ embedded in the Euclidean plane,
its {\bf chiral} graph is obtained
by rotating $\bm{G}$ through $180$ degrees
about the origin,
then reversing each edge vector.
While this process does not precisely produce the mirror image
(the meaning of the word ``chiral'' in chemistry),
it is faithful to the key idea.

\begin{theorem}
\label{thm:chiral}
If $\bm{G}$ is a Kirchhoff graph,
then so is its chiral.
\end{theorem}

\noindent
A Kirchhoff graph is a \textbf{self-chiral}
if and only if its chiral is itself.
In other words,
it is invariant under the chiral action.

\section{Finding Kirchhoff Graphs Using Uniformity}
\label{construction}

For uniform Kirchhoff graphs,
one can use their uniformity
as a basis for a backtracking exhaustive search construction algorithm.
Let a set of edge vectors $S$ and thus a row matrix $R=[qI|C]$ be given.
First a list is constructed of all the vertex cuts
that both lie in $\mbox{Row}(R)$
and have entries whose absolute values are no greater than
a given multiplicity bound $m_{\text{max}}$.
Then starting from an anchor vertex,
one considers whether or not the first entry on this list
might be able to be the vertex cut
for this base vertex,
and then whether it can be the vertex cut for each
new vertices that must be in the graph
given the initial base vertex cut,
without having the total number of times any edge vector
appears in the prospective graph exceed $m_{\text{max}}$.
Whenever this multiplicity bound is reached,
the most recent vertex cut is discarded,
and the next one on the list is considered in its place.
This process continues until either a uniform Kirchhoff graph is found,
or all entries from the list have been considered.
Using this algorithm,
one can find
all of the uniform Kirchhoff graphs
whose multiplicity do not exceed the multiplicity bounded $m_{\text{max}}$,
or it can be shown that no Kirchhoff graph exists for $S$
with edge multiplicity not exceeding $m_{\text{max}}$.

The brief outline above of our algorithm
is discussed in more detail in the next subsection.
Our code that implements this algorithm in java
can be found at \url{https://github.com/Jessica-Wang-Math/Kirchhoff.git} 

\subsection{Structure of the Algorithm}
\label{algo-structure}
What follows is a somewhat more detailed description
of the backtracking exhaustive search algorithm
for a given matrix $R$ and multiplicity $m_{\text{max}}$.

\begin{enumerate}

\item
Find all possible vertex cuts
with entries between $-m_{\text{max}}$ and $m_{\text{max}}$
by finding all linear combinations of the row vectors of $R$.
Let $\Lambda$ be the list of these vertex cuts
in an arbitrary order.
Initialize $\mathbb{T}$ as an empty list for us to add potential vertices to
as graph construction continues;
this serves as our to-do list.

\item
Place a starting or anchor vertex at the origin in $k$ dimensional Euclidean space.
Because of the way that $R$ and $N$ are defined,
all of the vertices will occur that integral coordinates.
In addition,
because every Kirchhoff graph is finite,
no vertex will occur at coordinates
whose sum is negative.
In other words,
no vertex will be below or behind the anchor vertex.
 
\item
Assign the first vertex cut from $\Lambda$
to the anchor vertex,
adding in the required edge vectors
and required incident vertices.
If any of these vertices have coordinates
whose sum is negative,
delete all of these edge vectors and incident vertices,
and consider the next vertex cut from $\Lambda$.
Otherwise,
add the neighboring incident vertices to the to-do list $\mathbb{T}$.

\item
Go to the next vertex $v_i$ in the graph (according to the order in $\mathbb{T}$)
and check whether it is already in $\mbox{Row}(R)$.
Assuming that it is not,
assign the first vertex cut from $\Lambda$ to it,
and check whether any of the new incident vertices
(1) have coordinates whose sum is negative,
or (2) result in the edge vector count for any edge vector exceeding $m_{\text{max}}$.
If either of these occur,
delete all of the new edge vectors and incident vertices,
and check the next vertex cut from $\Lambda$.
If not,
add all of the new incident vertices to $\mathbb{T}$
if they are not already on the list.
Notice that some of the new incident vertices may have
previously been removed from $\mathbb{T}$,
but the newly added vectors
may imply that their vertex cuts or no longer in $\mbox{Row}(R)$.
Delete $v_i$ from $\mathbb{T}$,
and go to the next vertex on $\mathbb{T}$ after $v_i$.

\item
When the final vertex cut on $\Lambda$
is eliminated at a vertex,
that vertex is abandoned,
and we move back to the previous vertex
which is placed back at the top of the to-do list.
For this previous vertex,
we consider the next vertex cut on $\Lambda$.
The new edges and incident vertices
required for this next vertex cut are added,
the to-do list is updated,
and the process moves back
to the previous step (4) in the algorithm.

\item
The process ends
either when there are no vertices left on the to-do list $\mathbb{T}$
(in which case a Kirchhoff graph is found,
and we consider the next vertex cut from $\Lambda$
at the anchor vertex),
or when the last possible vertex cut on $\Lambda$
is eliminated at the anchor vertex
(in which case no further Kirchhoff graphs exist,
and the entire process ends).
Thus the algorithm either finds all Kirchhoff graphs $\G$
with $m(\G)\le m_{\text{max}}$.

\end{enumerate}

\subsection{Kirchhoff Graph Examples Found by the Algorithm}

The algorithm discussed above
shows that the two Kirchhoff graphs shown in Figure~\ref{fig:F1}
are in fact the only ones
for $R_1$ with $m_{\text{max}}=2$.
Given $$R_2 = \begin{bmatrix}
~2 & 0 & 1 & 1~\\
~0 & 2 & 3 & 1~
\end{bmatrix}$$
and an upper bound multiplicity $m_{\text{max}}=6$,
the algorithm finds $16$ non-trivial Kirchhoff graphs,
as shown in Figure~\ref{fig:16-graphs}.
Notice that $m_{\text{max}}=6$ is the smallest multiplicity
for any Kirchhoff graphs associated with $R_2$.
The first $8$ Kirchhoff graphs are self-chirals,
while the rest are chiral pairs.
For another example,
let $$R_3=\begin{bmatrix}
~1 & 0 & 2 & 1~\\
~0 & 1 & 1 & 2~
\end{bmatrix}$$
and let $m_{\text{max}}=6$.
In this case,
our algorithm finds $4$ prime Kirchhoff graphs,
as shown in Figure~\ref{fig:4-graphs}.
Two form a chiral pair, and two are self-chirals.

\begin{figure}[htp]
    \centering
    \includegraphics[width=12cm]{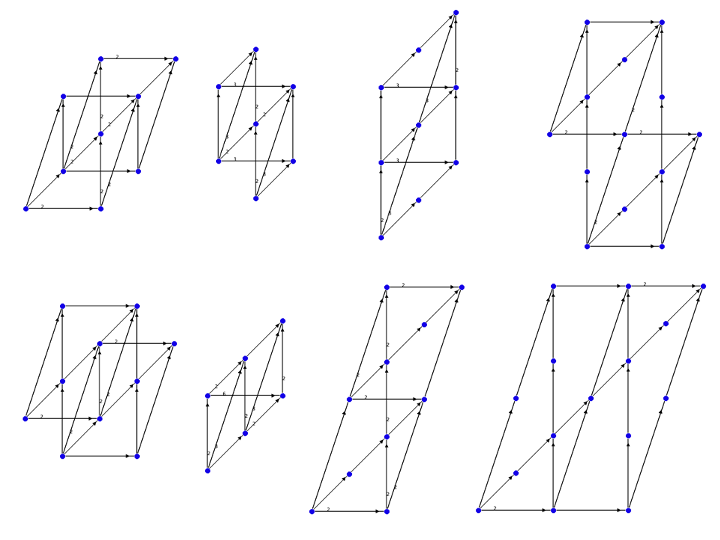}
    \includegraphics[width=12cm]{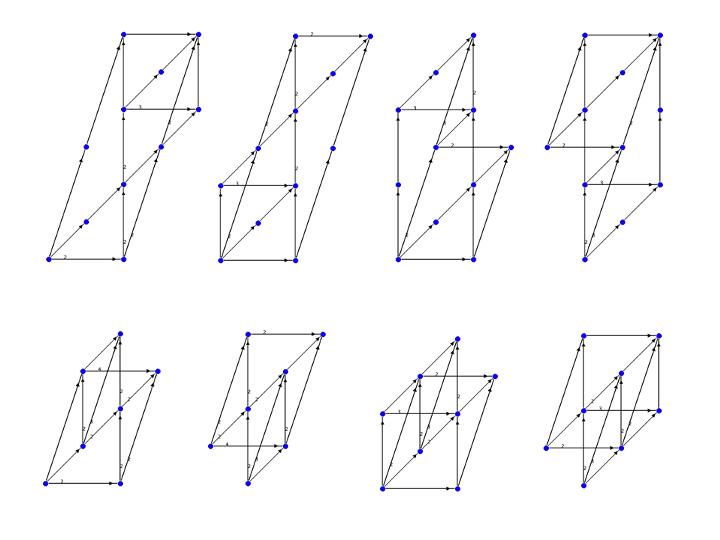}
    \caption{Sixteen prime Kirchhoff graphs for matrix $R_2$.
             The bottom eight are four chiral pairs;
             the top eight are self-chirals.}
    \label{fig:16-graphs}
\end{figure}

\begin{figure}[htp]
    \centering
    \includegraphics[height=8.0cm]{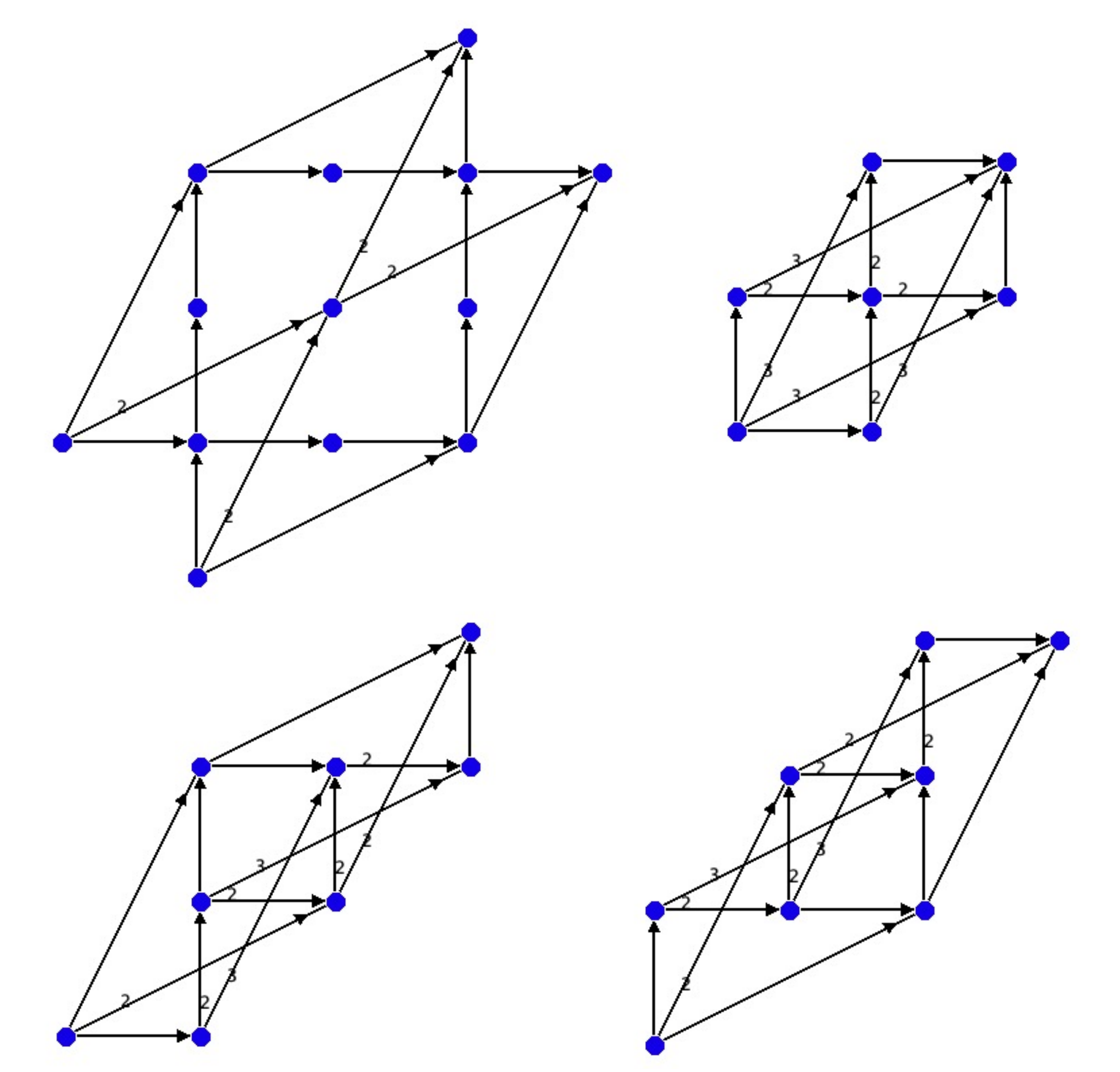}
    \caption{Four prime Kirchhoff graphs for matrix $R_3$.
             The bottom two are a chiral pair;
             the top two are self-chirals.}
    \label{fig:4-graphs}
\end{figure}

\section{Tiling of Kirchhoff Graphs}
\label{tiling}

This section considers the structure
of Kirchhoff graph families
like those shown above.
Doing this requires the concept of tiling,
which in turn requires the operations of addition and subtraction.

\begin{definition}
\label{addition}
Given two Kirchhoff graphs $\bb{G}{1},\bb{G}{2}$,
and a coordinate $\bm{x}\in\mathbb{Z}^k$,
define the {\bf sum} $(\bb{G}{1}+\bb{G}{2},\bm{x})$
as the union of $\bb{G}{1}$ with its anchor vertex placed at the origin
and $\bb{G}{2}$ with its anchor vertex placed at the coordinate $\bm{x}.$
Then $V(\bb{G}{1}+\bb{G}{2})=V(\bb{G}{1})\cup V(\bb{G}{2})$
and $m(\bb{G}{1}+\bb{G}{2})=m(\bb{G}{1})+m(\bb{G}{2})$.
So {\em all} copies of the edge vectors
from both $\bb{G}{1}$ and $\bb{G}{2}$
are present in their sum.
When it is unambiguous,
write $\bb{G}{1}+\bb{G}{2}$ as a shorthand.
Also for simplicity,
we only consider sums that preserve vector 2-connectivity.
In addition,
define the {\bf difference} $(\bb{G}{1}-\bb{G}{2},\bm{x})$
as the removal
of $\bb{G}{2}$
from $\bb{G}{1}$
assuming that there is a copy of $\bb{G}{2}$
as a subgraph of $\bb{G}{1}$
anchored at $\bm{x})$.
Collectively
the process of repeatedly adding and/or subtracting Kirchhoff graphs
is called a {\bf tiling},
and the resulting Kirchhoff graph
is called a {\bf tiling}.
\end{definition}

For the Kirchhoff graphs in Figure~\ref{fig:F1},
their sum $(\bb{F}{1}+\bb{F}{2},(1,1))$
is shown in Figure~\ref{fig:A1}.
\begin{figure}[bht]
    \centering
    \includegraphics[width=6cm]{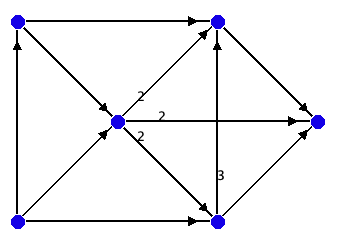}
    \caption{An addition example $(\bm{F}_1+\bm{F}_2,(1,1))$ for the row matrix $R_1$.}
    \label{fig:A1}
\end{figure}
Notice that the sum or difference of Kirchhoff graphs
must itself also be a Kirchhoff graph.
By tiling Kirchhoff graphs in various ways,
one can generate a wide variety of Kirchhoff graphs.

\begin{definition}
Given $N$ Kirchhoff graphs $\bb{G}{1},\ldots,\bb{G}{N}$,
the set of all Kirchhoff graphs
that can be constructed by tiling these $N$ Kirchhoff graphs is
$\langle\bb{G}{1},\ldots,\bb{G}{N}\rangle:=\{a_1\bb{G}{1}+\ldots+a_N\bb{G}{N}:a_i\in\Z\}$,
where $a_i\bb{G}{i}$ is any Kirchhoff graph that can be constructed
by tiling $\bb{G}{i}$,
using any anchor coordinate.
The $a_i$ can be negative
only when there exists $|a_i|$ copies of $\bb{G}{i}$ in the larger Kirchhoff graph.
\end{definition}

Now consider the principal definition of this section:
\begin{definition}
A Kirchhoff graph $\bm{G}$ is {\bf prime} if and only if
$\bm{G}$ has no nontrivial Kirchhoff subgraph decomposition.
In other words,
$\bm{G}$ cannot be written as $\bb{G}{1}+\bb{G}{2}$
where both $\bb{G}{1}$ and $\bb{G}{2}$ are nontrivial Kirchhoff graphs.
A Kirchhoff graph that is not prime is {\bf composite}.
\end{definition}

The Kirchhoff graphs shown in Figures~\ref{fig:F1}--\ref{fig:4-graphs}
are all prime;
the one in Figure~\ref{fig:A1} is of course composite.
Indeed it might seem that all Kirchhoff graphs that are tilings of smaller Kirchhoff graphs
will themselves be composite.
Perhaps surprisingly,
this is not the case:
Consider the two Kirchhoff graphs in Figure~\ref{fig:CP1}.
\begin{figure}[htb]
    \centering
    \includegraphics[width=5.7cm]{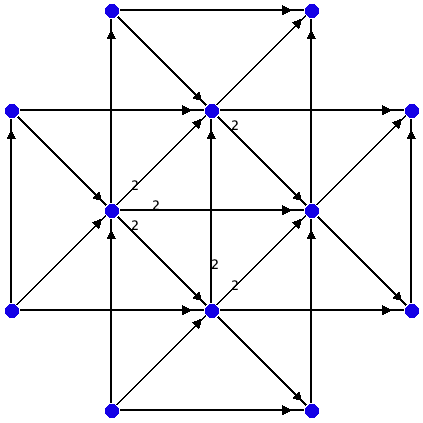}\qquad
    \includegraphics[width=5.7cm]{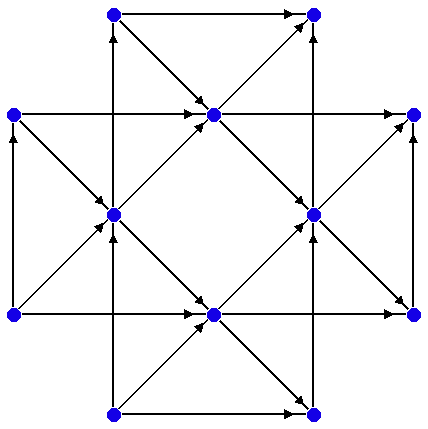}
    \caption{The Kirchhoff graph on the left is the composite $\bb{C}{1}:=4\bb{F}{1}$;
             the one on the left is the prime $\bb{P}{1}:=4\bb{F}{1}-\bb{F}{2}$.}
    \label{fig:CP1}
\end{figure}
The Kirchhoff graph on the left is a $\bb{C}{1}:=4\bb{F}{1}$
(a tiling by addition of four copies of $\bb{F}{1}$);
it is clearly composite.
The Kirchhoff graph on the right is $\bb{P}{1}:=4\bb{F}{1}-\bb{F}{2}$
(the copy of $\bb{F}{2}$ in the middle has been removed);
it is prime.
This can be seen because if any edge vector is removed,
the vertex cuts for incident vertices will no longer lie in $\mbox{Row}(R_1)$,
and the only way to return all the vertex cuts to $\mbox{Row}(R_1)$
is to remove all the edge vectors,
meaning that there is no Kirchhoff subgraph.

The above example makes clear
that composite Kirchhoff graphs
may not have unique prime decompositions.
In this example,
$\bb{C}{1}=4\bb{F}{1}=\bb{P}{1}+\bb{F}{2}$.
On the other hand,
there are infinity many prime Kirchhoff graphs.

\begin{theorem}\label{InftyPrime}
For the two prime Kirchhoff graphs $\bb{F}{1}$ and $\bb{F}{2}$
for the matrix
$$
R_1=\left[\begin{array}{rrrr}
~2 &~0 &~1&~1 \\
~0 &~2 &~1&-1
\end{array}\right]
$$
with $m_{\text{max}}= 2$
{\em (see Figure~\ref{fig:F1})},
$\langle\bb{F}{1},\bb{F}{2}\rangle$ contains infinitely many prime Kirchhoff graphs.
\end{theorem}
\begin{proof}
Constructing arbitrarily large prime Kirchhoff graphs
that are in $\langle\bb{F}{1},\bb{F}{2}\rangle$
is a simple extension
of the construction of the prime Kirchhoff graph
in Figure~\ref{fig:CP1}. 
\begin{figure}[htb]
    \centering
    \includegraphics[width=5.7cm]{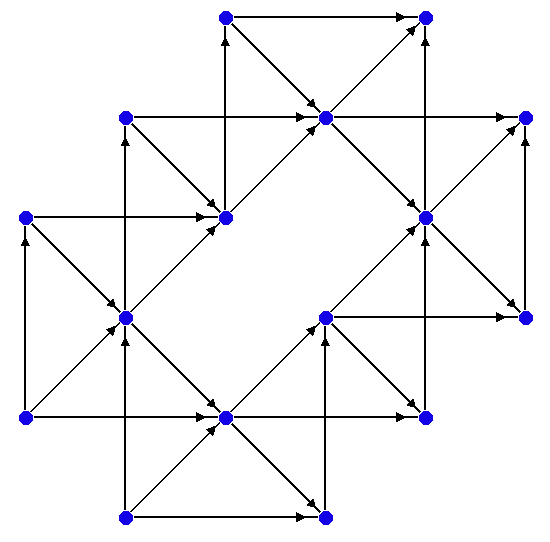}\qquad
    \includegraphics[width=5.7cm]{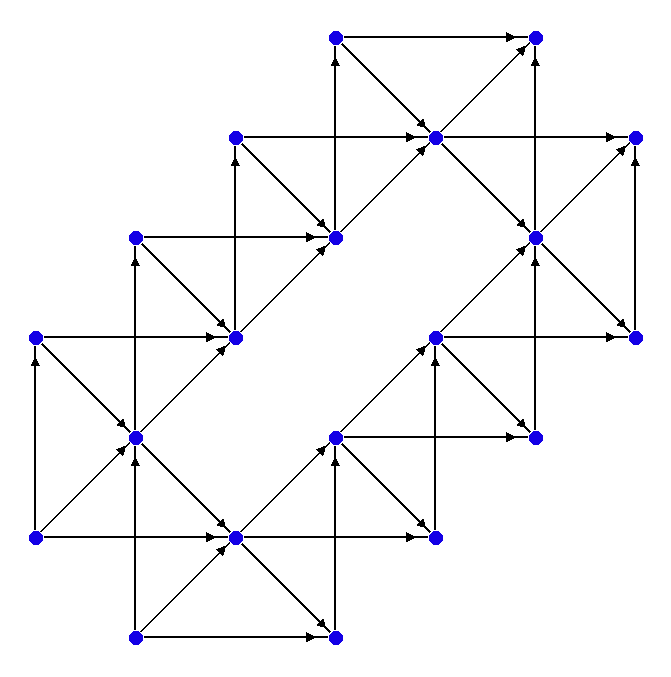}
    \caption{Two larger prime Kirchhoff graphs
             generated by tiling for $R_1$.}
    \label{fig:P2P3}
\end{figure}
Simply form $6\bb{F}{1}-2\bb{F}{2}$
by adding two more copies of $\bb{F}{1}$ 
to the prime Kirchhoff graph
in Figure~\ref{fig:CP1},
and then by subtracting $\bb{F}{2}$ from the middle.
Notice that the addition is exactly
what is needed
to create the copy of $\bb{F}{2}$
and thus make possible the subtraction.
Then repeat this tiling
until the prime Kirchhoff graph
of the desired size is achieved.
\end{proof}

\begin{remark}
Although it is hard to describe in general,
this sort of construction
of prime Kirchhoff graphs
of arbitrary size
would seem to be possible
in any $\langle\bb{K}{1},\bb{K}{2}\rangle$
for any pair of prime Kirchhoff graphs
of minimal multiplicity
for a given row matrix $R$,
particularly when the generating Kirchhoff graphs
$\bb{K}{1}$ and $\bb{K}{2}$ are self-chiral.
\end{remark}

\section{Fundamental Graphs, Tiling Structure}
\label{structure}
One final important implication of this work
is a beginning of an understanding
of the structure of Kirchhoff graph tilings.
Notice that for a given set of edge vectors $S$
and the corresponding row matrix $R$,
there is a minimum edge multiplicity number $m^*$
below which there can be no Kirchhoff graphs.
With the assumptions that no edge vector is a constant multiple of another,
the smallest possible Kirchhoff graph
is a vector triangle
having $n=3$, $k=2$ and $m^*=1$.
For $m=m^*$ there will be a finite number of Kirchhoff graphs,
and each of these will be prime
since any Kirchhoff subgraph
would have fewer than $m^*$ edge vectors.
It is frequently the case, however,
that one or more of these graphs
is in fact a tiling of others.
An example of this is shown in Figure~\ref{fig:4-graphs}
for $R_3$
where any one of the four Kirchhoff graphs
is a tiling of the other three.
This structure leads to one additional definition,
that of {\em fundamental Kirchhoff graphs}.

\begin{definition}
A {\bf fundamental set} for $R$ and $S$ is a minimal generating set
in term of tiling
with respect to both multiplicity and cardinality.
Members of a fundamental set are called {\bf fundamental Kirchhoff graphs}.
\end{definition}
\noindent
For $R_1$,
both of the Kirchhoff graphs in Figure~\ref{fig:F1} are fundamental.
For $R_3$,
any three of the Kirchhoff graphs in Figure~\ref{fig:4-graphs} are fundamental. 
For $R_2$,
no more than twelve of the sixteen are fundamental.
Based on our current computations,
all known larger Kirchhoff graphs are tilings of the graphs
in the fundamental set,
and thus all have an edge vector multiplicity
that is an integral multiple of $m^{\*}$,
though this final point remains unproven.

\section{Open Questions}
\label{OQ}
The work completed so far
on prime, composite and fundamental Kirchhoff graphs
has lead to a number of open questions:
\begin{itemize}
    \item Given a matrix $R$,
          is it possible to predict $m^*$,
          the smallest $m$
          such that nontrivial Kirchhoff graphs exist,
          without constructing the graphs?
    \item Given a matrix $R$ and minimal multiplicity $m^*$,
          how many prime Kirchhoff graphs are there?
          Our computations suggest that
          when there are two linearly independent edge vectors ($k=2$),
          this number will be some power of 2.
    \item Is there a matrix $R$
          with is a larger Kirchhoff graph $\bm{F}^+\not\in\langle\bb{F}{1},\ldots,\bb{F}{\ell}\rangle$?
    \item Is there a condition on $R$ for the existence of chiral pairs?
\end{itemize}

\section*{Acknowledgement}
The authors wish to thank Padraig \'O Cath\'ain, Randy Paffenroth and Brigitte Servatius
for many helpful discussions of this work.
\bibliographystyle{unsrt}
\bibliography{ref}

\end{document}